\documentclass[12pt, reqno]{amsart}
\usepackage{amsmath, amsthm, amscd, amsfonts, amssymb, graphicx, color}
\usepackage[bookmarksnumbered, colorlinks, plainpages]{hyperref}

\usepackage{mathrsfs}
\usepackage{enumerate}

\newtheorem{theorem}{Theorem}[section]
\newtheorem{lemma}[theorem]{Lemma}
\newtheorem{proposition}[theorem]{Proposition}
\newtheorem{corollary}[theorem]{Corollary}
\theoremstyle{definition}
\newtheorem{definition}[theorem]{Definition}
\newtheorem{example}[theorem]{Example}

\theoremstyle{remark}
\newtheorem{remark}[theorem]{Remark}

\numberwithin{equation}{section}

\newcommand{\e}{\varepsilon}
\newcommand{\wh}{\widehat}
\newcommand{\U}{\mathbf{1}}
\newcommand{\z}{\mathbf{0}}

\newcommand{\tp}{\otimes}
\newcommand{\ptp}{\mathbin{\hat{\otimes}}}

\DeclareMathOperator{\Inv}{Inv}
\DeclareMathOperator{\Lip}{Lip}
\DeclareMathOperator{\rad}{rad}

\newcommand{\C}{\mathbb{C}}
\newcommand{\T}{\mathbb{T}}

\newcommand{{\fA}}{{\mathfrak{{A}}}}
\newcommand{{\fB}}{{\mathfrak{{B}}}}

\newcommand{{\cA}}{{\mathcal{{A}}}}
\newcommand{{\cB}}{{\mathcal{{B}}}}
\newcommand{{\cE}}{{\mathcal{{E}}}}
\newcommand{{\cI}}{{\mathcal{{I}}}}
\newcommand{{\cM}}{{\mathcal{{M}}}}
\newcommand{{\cX}}{{\mathcal{{X}}}}

\newcommand{{\scrA}}{\mathscr{A}}
\newcommand{\A}{\scrA}

\newcommand{\Sp}{\text{\textsc{sp}}}
\newcommand{\vsp}{\vec{\text{\textsc{sp}}}}

\newcommand{\mis}{\mathfrak{M}}
\newcommand{\I}{\mathrm{i}}

\newcommand{\set}[1]{\{#1\}}
\newcommand{\bigset}[1]{\bigl\{ #1 \bigr\}}

\newcommand{\Bigabs}[1]{\Bigl\lvert #1 \Bigr\rvert}
\newcommand{\bigprn}[1]{\bigl( #1 \bigr)}

\newcommand{\enorm}{\lVert\,\cdot\,\rVert}
\newcommand{\norm}[1]{\lVert #1 \rVert}

\newcommand{\Bignorm}[1]{\Bigl\lVert #1 \Bigr\rVert}

\newcommand{\lVertt}{\lvert\mspace{-2mu}\lvert\mspace{-2mu}\lvert}
\newcommand{\rVertt}{\rvert\mspace{-2mu}\rvert\mspace{-2mu}\rvert}
\newcommand{\enormm}{\lVertt\, \cdot \,\rVertt}
\newcommand{\normm}[1]{\lVertt #1 \rVertt}

\begin{document}
\setcounter{page}{1}

\title[Characters on Vector-valued Function Algebras]%
{Vector-valued characters on Vector-valued Function Algebras}

\author[M. Abtahi]{Mortaza Abtahi}
\address{Mortaza Abtahi\newline
School of Mathematics and Computer Sciences,
Damghan University, Damghan,
P.O.BOX 36715-364, Iran}

\email{\textcolor[rgb]{0.00,0.00,0.84}{abtahi@du.ac.ir; mortaza.abtahi@gmail.com}}




\date{Received: xxxxxx; Revised: yyyyyy; Accepted: zzzzzz.}

\begin{abstract}
  Let $A$ be a commutative Banach algebra and $X$ be a compact space. The
  class of Banach $A$-valued function algebras on $X$ consists of subalgebras
  of $C(X,A)$ with certain properties. We introduce the notion of $A$-characters
  on an $A$-valued function algebra $\A$ as homomorphisms from $\A$ into $A$
  that basically have the same properties as the evaluation homomorphisms $\cE_x:f\mapsto f(x)$,
  with $x\in X$. For the so-called natural $A$-valued function algebras, such as
  $C(X,A)$ and $\Lip(X,A)$, we show that $\cE_x$ ($x\in X$) are the only
  $A$-characters. Vector-valued characters are utilized to identify vector-valued
  spectrums. When $A=\C$, Banach $A$-valued function algebras reduce to
  Banach function algebras, and $A$-characters reduce to characters.
\end{abstract}
\maketitle

\section{Introduction and Preliminaries}
\label{sec:intro}

In this paper, we consider commutative unital Banach algebras over the complex field $\C$.
For more information on the theory of commutative Banach algebras,
see the texts \cite{BD,Dales,CBA,Zelazko}.

\subsection{Characters}
Let $A$ be a commutative Banach algebra with identity $\U$.
Every nonzero homomorphism $\phi:A\to\C$ is called a \emph{character} of $A$.
The set of all characters of $A$ is denoted by $\mis(A)$. It is well-known that
$\mis(A)$ is nonempty and its elements are automatically continuous
\cite[Lemma 2.1.5]{CBA}. For every $a\in A$, define $\hat a:\mis(A)\to\C$
by $\hat a(\phi)=\phi(a)$. The \emph{Gelfand topology} of $\mis(A)$ is then
the weakest topology on $\mis(A)$ that makes every $\hat a$ a continuous function.
Equipped with the Gelfand topology, $\mis(A)$ is a compact Hausdorff space.
The set $\hat A=\set{\hat a:a\in A}$ is a subalgebra of
$C(\mis(A))$ and the \emph{Gelfand transform} $A\to\hat A$, $a\mapsto\hat a$,
is an algebra homomorphisms with $\rad A=\set{a:\hat a=0}$ as its kernel.
The algebra $A$ is called \emph{semisimple} if $\rad A=\z$.

\subsection{Complex Function Algebras}
Let $X$ be a compact Hausdorff space and let $C(X)$ be the algebra of all continuous
functions $f:X\to\C$ equipped with the uniform norm $\|f\|_X=\sup\set{|f(x)|:x\in X}$.
We recall that a subalgebra $\fA$ of $C(X)$ that separates the points of $X$ and
contains the constant functions is a \emph{function algebra} on $X$.
A function algebra $\fA$ equipped with some algebra norm $\enorm$ such that
$\|f\|_X\leq \|f\|$, for all $f\in \fA$, is a \emph{normed function algebra}.
A complete normed function algebra is a \emph{Banach function algebra}.
If the norm of a Banach function algebra $\fA$ is equivalent to the uniform norm $\enorm_X$,
then $\fA$ is a \emph{uniform algebra}.

Let $\fA$ be a Banach function algebra on $X$. For every $x\in X$, the mapping
$\e_x:\fA\to\C$, $\e_x(f)=f(x)$, is a character of $\fA$, and
the mapping $J:X\to\mis(\fA)$, $x\mapsto\e_x$, imbeds $X$ homeomorphically as
a compact subset of $\mis(\fA)$. When $J$ is surjective, one calls $\fA$
\emph{natural} \cite[Chapter 4]{Dales}. For instance, $C(X)$ is a natural uniform algebra.
We remark that every semisimple commutative unital Banach algebra $A$ can
be seen (through its Gelfand transform) as a natural Banach function
algebra on $X=\mis(A)$. For more on function algebras
see, for example, \cite{Gamelin-UA,Leibowitz}.

\subsection{Vector-valued Function Algebras}
Let $(A,\enorm)$ be a commutative unital Banach algebra. The set of
all $A$-valued continuous functions on $X$ is denoted
by $C(X,A)$. Note that $C(X)=C(X,\C)$. For $f,g\in C(X,A)$ and $\alpha\in\C$,
the pointwise operations $\alpha f+g$ and $fg$ are defined in the obvious way.
The \emph{uniform norm} of each function $f\in C(X,A)$ is defined by
$\|f\|_X=\sup\set{\|f(x)\|:x\in X}$. In this setting, $\bigl(C(X,A),\enorm_X\bigr)$
is a commutative unital Banach algebra.

Starting by Yood \cite{Yood}, in 1950, the character space of $C(X,A)$
has been studied by many authors. In 1957, Hausner \cite{Hausner} proved that
$\tau$ is a character of $C(X,A)$ if, and only if,
there exist a point $x\in X$ and a character $\phi\in\mis(A)$ such that
$\tau(f) = \phi(f(x))$, for all $f\in C(X,A)$, whence $\mis(C(X,A))$ is
homeomorphic to $X \times \mis(A)$. Recently, in \cite{Ab-Ab}, other characterizations
of maximal ideals in $C(X,A)$ have been presented.

Analogous with complex function algebras, vector-valued function
algebras are defined (Definition \ref{dfn:vector-valued-function-algebras}).

\subsection{Vector-valued Characters}

Let $\A$ be an $A$-valued function algebra on $X$. For each $x\in X$, consider the
evaluation homomorphism $\cE_x:\A\to A$, $x\mapsto f(x)$. These $A$-valued homomorphisms
are included in a certain class of homomorphisms that will be introduced and studied in
Section \ref{sec:vector-valued-characters} under the name of \emph{vector-valued characters}.
More precisely,
a homomorphism $\Psi:\A\to A$ is called an \emph{$A$-character} if $\Psi(\U)=\U$
and $\phi(\Psi f)=\Psi(\phi\circ f)$, for all $f\in \A$ and $\phi\in\mis(A)$.
See Definition \ref{dfn:vector-valued-character}.
The set of all $A$-characters of $\A$ will be denoted by $\mis_A(\A)$.
Note that when $A=\C$, we have $\mis_A(\A)=\mis_\C(\A)=\mis(\A)$.

An application of vector-valued characters are presented in the forthcoming
paper \cite{abtahi-farhangi} to identify the vector-valued spectrum of
functions $f\in\A$. It is known that the spectrum of an element $a\in A$
is equal to $\Sp(a)=\set{\phi(a):\phi\in \mis(A)}$. In \cite{abtahi-farhangi}
the $A$-valued spectrum $\vsp_A(f)$ of functions $f\in \A$ are studied and
it is proved that, under certain conditions,
\[
  \vsp_A(f) = \set{\Psi(f):\Psi\in \mis_A(\A)}.
\]

\subsection{Notations and conventions}
Since, in this paper, we are dealing with different types of functions and algebras,
ambiguity may arise. Hence, a clear declaration of notations and conventions is
given here.

\begin{enumerate}[(1)]

  \item Throughout the paper, $X$ is a compact Hausdorff space, and
  $A$ is a commutative \emph{semisimple} Banach algebra with identity $\U$. The set of
  invertible elements of $A$ is denoted by $\Inv(A)$.

  \item If $f:X\to\C$ is a function and $a\in A$, we write $fa$ to denote the $A$-valued function
  $X\to A$, $x\mapsto f(x)a$.
  If $\fA$ is a complex-valued function algebra on $X$, we let $\fA A$ be
  the linear span of $\set{fa:f\in\fA,\,a\in A}$. Hence, an element $f\in \fA A$
  is of the form $f=f_1a_1+\dotsb+f_na_n$ with $f_j\in \fA$ and $a_j\in A$.

  \item Given an element $a\in A$, we use the same notation $a$ for the constant
  function $X\to A$ given by $a(x)=a$, for all $x\in X$, and consider $A$ as a closed
  subalgebra of $C(X,A)$.
  Since we assume that $A$ has an identity $\U$, we identify $\C$ with
  the closed subalgebra $\C\U$ of $A$, and thus every continuous
  function $f:X\to\C$ can be seen as the continuous A-valued function
  $X\to A$, $x\mapsto f(x)\U$; we use the same notation $f$ for
  this $A$-valued function, and regard $C(X)$ as a closed subalgebra of $C(X,A)$.

  \item To every continuous function $f:X\to A$, we correspond the function
  \[
     \tilde f:\mis(A)\to C(X),\quad \tilde f(\phi)=\phi\circ f.
  \]
  If $\cI$ is a family of continuous $A$-valued functions on $X$, we define
  \[
    \phi[\cI]=\set{\phi\circ f: f\in\cI}=\set{\tilde f(\phi):f\in\cI}.
  \]
\end{enumerate}

\section{Vector-valued Function Algebras}
\label{sec:vector-valued-function-algebras}

In this section, we introduce the notion of vector-valued function algebras
as subalgebras of $C(X,A)$ with certain properties.
Vector-valued function algebras have been defined and studied by others
(e.g.\ \cite{Nikou-Ofarrell}).

\begin{definition}
\label{dfn:vector-valued-function-algebras}
  Let $X$ be a compact Hausdorff space, and $(A,\enorm)$ be a commutative unital
  Banach algebra. An \emph{$A$-valued function algebra} on $X$ is a subalgebra $\A$
  of $C(X,A)$ such that (1) $\A$ contains all the constant functions $X\to A$, $x\mapsto a$,
  with $a\in A$, and (2) $\A$ separates the points of $X$ in the sense that,
  for every pair $x,y$ of distinct points in $X$, there exists $f \in A$ such that
  $f(x) \neq f(y)$. A \emph{normed $A$-valued function algebra} on $X$ is an $A$-valued
  function algebra $\A$ on $X$ endowed with some algebra norm $\enormm$ such that
  the restriction of $\enormm$ to $A$ is equivalent to the original norm $\enorm$ of $A$,
  and $\|f\|_X \leq \normm{f}$, for every $f\in \A$.
  A complete normed $A$-valued function algebra is called
  a \emph{Banach $A$-valued function algebra}. A Banach $A$-valued function algebra
  $\A$ is called an \emph{$A$-valued uniform algebra} if the given norm of $\A$ is equivalent
  to the uniform norm $\enorm_X$.
\end{definition}

If there is no risk of confusion, instead of $\enormm$, we use the same notation $\enorm$ for the norm
of $\A$.

Let $\A$ be an $A$-valued function algebra on $X$. For every $x\in X$, define
$\cE_x:\A\to A$ by $\cE_x(f)=f(x)$. We call $\cE_x$ the \emph{evaluation homomorphism}
at $x$. Our definition of Banach $A$-valued function algebras implies that every
evaluation homomorphism $\cE_x$ is continuous. As it is mentioned in \cite{Nikou-Ofarrell},
if the condition $\|f\|_X \leq \normm{f}$, for all $f\in\A$, is replaced by the requirement
that every evaluation homomorphism $\cE_x$ is continuous, then
one can find some constant $M$ such that
\[
  \|f\|_X \leq M \normm{f} \quad (f\in\A).
\]

\subsection{Admissible algebras}
Given a complex-valued function algebra $\fA$ and an $A$-valued function algebra $\A$
on $X$, according to \cite[Definition 2.1]{Nikou-Ofarrell}, the quadruple $(X,A,\fA,\A)$
is \emph{admissible} if $\fA$ is natural, $\fA A\subset \A$, and
\[ \set{\phi\circ f: \phi\in\mis(A), f\in \A} \subset \fA. \]
Taking this into account, we make the following definition.

\begin{definition}
  An $A$-valued function algebra $\A$ is said to be \emph{admissible} if
  \begin{equation}\label{eqn:admissible-A-valued-FA}
     \set{\phi\circ f: \phi\in\mis(A), f\in \A} \subset \A.
  \end{equation}
\end{definition}

When $\A$ is admissible, we set $\fA=C(X)\cap\A$ which is the subalgebra
of $\A$ consisting of all complex-valued functions in $\A$ and forms a complex-valued
function algebra by itself. Note that, in this case, $\fA=\phi[\A]$, for all $\phi\in\mis(A)$.

\begin{example}
  Let $\fA$ be a complex-valued function algebra on $X$. Then $\fA A$ is an admissible
  $A$-valued function algebra on $X$. Hence, the uniform closure of $\fA A$ in $C(X,A)$ is
  an admissible $A$-valued uniform algebra
  (Proposition \ref{prop:uniform-closure-of-function-algebras} below).
\end{example}

Another example of an admissible Banach $A$-valued function algebra is the
$A$-valued Lipschitz function algebra $\Lip(X,A)$, where $X$ is a compact metric space.
Recently, in \cite{Esmaeili-Mahyar}, the character space and \'Silov boundary of
$\Lip(X,A)$ has been studied. For certain properties of vector-valued Lipschitz spaces,
see \cite{Cao} and \cite{Johnson-thesis,Johnson-art}.

\bigskip
The following shows that not all vector-valued function algebras are admissible.

\begin{example}
  Let $K$ be a compact subset of $\C$ which is not polynomially convex
  so that $P(K)\neq R(K)$. For example, let $K=\T$ be the unit circle. Set
  \[
    \A=\set{(f_p,f_r): f_p\in P(K), f_r\in R(K)}.
  \]

  Then $\A$ is a uniformly closed subalgebra of $C(K,\C^2)$ which contains all the constant
  functions $(\alpha,\beta)\in\C^2$ and separates the points of $K$.
  Hence $\A$ is a $\C^2$-valued uniform algebra on $K$.
  Let $\U=(1,1)$ be the unit element of $\C^2$, and let $\pi_1$ and $\pi_2$
  be the coordinate projections of $\C^2$. Then $\mis(\C^2)=\set{\pi_1,\pi_2}$, and
  \begin{align*}
    \pi_1[\A]\U & = \set{f\U:f\in P(K)} = \set{(f,f):f\in P(K)}, \\
    \pi_2[\A]\U & = \set{f\U:f\in R(K)} = \set{(f,f):f\in R(K)}.
  \end{align*}
  We see that $\pi_1[\A]\subset\A$ while $\pi_2[\A]\not\subset\A$.
  Hence $\A$ is not admissible.
\end{example}

\begin{proposition}
\label{prop:uniform-closure-of-function-algebras}
  Suppose $\A$ is an admissible $A$-valued function algebra on $X$.
  Then the uniform closure $\bar\A$ is an admissible $A$-valued uniform algebra on $X$ and
  $\bar\fA=C(X)\cap\bar\A$.
\end{proposition}

\begin{proof}
  Recall that $\fA=C(X)\cap \A$. The fact
  that $\bar\A$ is an $A$-valued uniform algebra is clear. The inclusion
  $\bar\fA\subset C(X)\cap\bar\A$ is also obvious. Take $f\in C(X)\cap\bar\A$.
  Then, there exists a sequence $\set{f_n}$ of $A$-valued functions in $\A$ such that
  $f_n\to f$ uniformly on $X$. For some $\phi\in\mis(A)$, take $g_n=\phi\circ f_n$.
  Then $g_n\in\fA$ and, since $f=\phi\circ f$, we have
  \[
    \norm{g_n-f}_X = \norm{\phi\circ f_n - \phi\circ f}_X
    \leq \norm{f_n-f}_X \to 0.
  \]
  Therefore, $g_n\to f$ uniformly on $X$ and thus $f\in\bar\fA$.
\end{proof}

\subsection{Certain vector-valued uniform algebras}

  Let $\fA_0$ be a complex function algebra on $X$, and let $\fA$ and $\A$ be the uniform
  closures of $\fA_0$ and $\fA_0A$, in $C(X)$ and $C(X,A)$, respectively. Then $\A$
  is an admissible $A$-valued uniform algebra on $X$ with $\fA=C(X)\cap \A$.
  The algebra $\A$ is isometrically isomorphic to the injective tensor product
  $\fA \ptp_\epsilon A$ (cf. \cite[Proposition 1.5.6]{CBA}). To see this, let
  $T:\fA\tp A \to \A$ be the unique linear mapping, given by \cite[Theorem 42.6]{BD},
  such that $T(f\tp a)(x) = f(x)a$, for all $x\in X$.
  Let $\fA^*_1$ and $A^*_1$ denote the closed unit ball of $\fA^*$ and $A^*$,
  respectively, and let $\enorm_\epsilon$ denote the injective tensor norm. Then
  \begin{align*}
    \Bignorm{T\Bigl(\sum_{i=1}^n f_i\tp a_i\Bigr)}_X
    & = \sup_{x\in X} \Bignorm{\sum_{i=1}^n f_i(x)a_i}
      = \sup_{x\in X} \sup_{\nu\in A^*_1}
        \Bigabs{\sum_{i=1}^n f_i(x)\nu(a_i)} \\
    & = \sup_{\nu\in A^*_1} \Bignorm{\sum_{i=1}^n f_i(\cdot) \nu(a_i)}_X
      = \sup_{\nu\in A^*_1}\sup_{\mu\in\fA^*_1}
        \Bigabs{\mu\Bigl(\sum_{i=1}^n f_i(\cdot)\nu(a_i)\Bigr)} \\
    & = \sup_{\mu\in\fA^*_1}\sup_{\nu\in A^*_1}
        \Bigabs{\sum_{i=1}^n \mu(f_i) \nu(a_i)}
      = \Bignorm{\sum_{i=1}^n f_i\tp a_i}_\epsilon.
  \end{align*}

  Hence $T$ extends to an isometry $\overline T$ from $\fA\ptp_\epsilon A$ into $\A$.
  Since the range of $T$ contains $\fA_0A$, which is dense in $\A$,
  the range of $\overline T$ is the whole of $\A$. We remark that,
  by a theorem of Tomiyama \cite{Tomiyama}, the maximal ideal space $\mis(\A)$
  of $\A$ is homeomorphic to $\mis(\fA)\times\mis(A)$.

Now, let $K$ be a compact subset of $\C$. Associated with $K$
there are three vector-valued uniform algebras in which
we are interested. Let $P_0(K,A)$ be the algebra of the restriction to $K$
of all polynomials $p(z)=a_nz^n+\dotsb+a_1z+a_0$ with coefficients
$a_0,a_1,\dotsc,a_n$ in $A$. Let $R_0(K,A)$ be the algebra
of the restriction to $K$ of all rational functions of the form
$p(z)/q(z)$, where $p(z)$ and $q(z)$ are polynomials with coefficients
in $A$, and $q(\lambda) \in \Inv(A)$, whenever $\lambda\in K$. And
let $H_0(K,A)$ be the algebra of all $A$-valued functions on $K$
having a holomorphic extension to a neighbourhood of $K$.

When $A=\C$, we drop $A$ and write $P_0(K)$, $R_0(K)$ and $H_0(K)$.
Their uniform closures in $C(K)$, denoted by $P(K)$, $R(K)$ and $H(K)$, are
complex uniform algebras (for more on complex uniform algebras, see
\cite{Gamelin-UA} or \cite{Leibowitz}).

The algebras $P_0(K,A)$, $R_0(K,A)$ and $H_0(K,A)$ are $A$-valued function algebras
on $K$, and their uniform closures in $C(K,A)$, denoted by $P(K,A)$, $R(K,A)$ and $H(K,A)$,
are $A$-valued uniform algebras. It obvious that
\[
  P(K,A)\subset R(K,A) \subset H(K,A).
\]

\medskip
Every polynomial $p(z)=a_nz^n+\dotsb+a_1z+a_0$ in $P_0(K,A)$ can be
written as $p(z)=p_0(z)a_0+p_1(z)a_1+\dotsb+p_n(z)a_n$,
where $p_0,p_1,\dotsc,p_n$ are polynomials in $P_0(K)$.
Thus $P_0(K,A)=P_0(K)A$. The above discussion shows that $P(K,A)$ is
isometrically isomorphic to $P(K)\ptp_\e A$. The maximal ideal space
of $P(K)$ is homeomorphic to $\hat K$, the polynomially
convex hull of $K$ (e.g.\ \cite[Section 5.2]{Leibowitz}).
The maximal ideal space of $P(K,A)$ is, therefore, homeomorphic to $\hat K\times \mis(A)$;
see also \cite{Nikou-Ofarrell}.

\medskip
Runge's classical approximation theorem states that if $\Lambda$ is a subset of $\C$
such that $\Lambda$ has nonempty intersection with each bounded component of
$\C\setminus K$, then every function $f\in H_0(K)$ can be approximated uniformly
on $K$ by rational functions with poles only among the points of $\Lambda$ and at infinity
(\cite[Theorem 7.7]{BD}). In particular, $R(K)=H(K)$.
The following is a version of Runge's theorem
for vector-valued functions.

\begin{theorem}[Runge]
  Let $K$ be a compact subset of $\C$, and let $\Lambda$ be a subset of
  $\C\setminus K$ having nonempty intersection with each bounded component
  of $\C\setminus K$. Then every function $f\in H_0(K,A)$ can be approximated
  uniformly on $K$ by $A$-valued rational functions of the form
  \begin{equation}\label{eqn:r(z)}
    r(z)=r_1(z)a_1+r_2(z)a_2+ \dotsb +r_n(z)a_n,
  \end{equation}
  where $r_i(z)$, for $1\leq i \leq n$, are rational functions in $R_0(K)$
  with poles only among the points of $\Lambda$ and at infinity,
  and $a_1,a_2,\dotsc,a_n\in A$.
\end{theorem}

\begin{proof}
  Take $f\in H_0(K,A)$. Then, there exists an open set $D$ such that $K\subset D$ and
  $f:D\to A$ is holomorphic. We use the same notation as in
  \cite{BD}. We let $E$ be a punched disc envelope for $(K,D)$;
  see \cite[Definition 6.2]{BD}. The Cauchy theorem and
  the Cauchy integral formula are also valid for Banach space valued holomorphic
  functions (see the remark after Corollary 6.6 in \cite{BD} and
  \cite[Theorem 3.31]{Rudin-FA}). Therefore,
  \[
    f(z) = \frac1{2\pi \I}\int_{\partial E} \frac{f(s)}{s-z}ds \quad (z\in K).
  \]

  \noindent
  Then, by \cite[Proposition 6.5]{BD}, one can write
  \begin{equation}\label{eqn:series-expansion}
    f(z)=\sum_{n=0}^\infty \alpha_n(z-z_0)^n +
        \sum_{j=1}^m \sum_{n=1}^\infty \frac{\beta_{jn}}{(z-z_j)^n} \quad (z\in K),
  \end{equation}
  where $z_0\in \C$, $z_1,z_2,\dotsc,z_m\in \C\setminus K$, and the coefficients
  $\alpha_n, \beta_{jn}$ belong to $A$. Note that the series in \eqref{eqn:series-expansion}
  converges uniformly on $K$.

  So far, we have seen that $f$ can be approximated uniformly on $K$ by $A$-valued rational
  functions of the form \eqref{eqn:r(z)}, where $r_i(z)$, for $1\leq i \leq n$, are rational
  functions in $R_0(K)$ with poles just outside $K$. Using Runge's classical theorem,
  each $r_i(z)$ can be approximated uniformly on $K$ by rational functions
  with poles only among the points of $\Lambda$ and at infinity.
  Hence, we conclude that $f$ can be approximated
  uniformly on $K$ by rational functions of the form \eqref{eqn:r(z)}
  with preassigned poles.
\end{proof}

As a consequence of the above theorem, we see that
the uniform closures of $R_0(K)A$, $H_0(K)A$, $R_0(K,A)$ and $H_0(K,A)$
are all the same. In particular,
\[ R(K,A)=H(K,A). \]

\begin{corollary}
  The algebra $R(K,A)$ is isometrically isomorphic to $R(K)\ptp_\e A$ and,
  therefore, $\mis(R(K,A))$ is homeomorphic to $K\times \mis(A)$.
\end{corollary}

We remark that the equality $\mis(R(K,A))=K\times \mis(A)$ is proved
in \cite{Nikou-Ofarrell}. The authors, however, did not notice
the equality $R(K,A)=R(K)\ptp_\e A$.

\section{Vector-valued Characters}
\label{sec:vector-valued-characters}

Let $\A$ be a Banach $A$-valued function algebra on $X$, and consider the point evaluation
homomorphisms $\cE_x:\A\to A$. These kind of homomorphisms enjoy the following properties:
\begin{itemize}
  \item $\cE_x(a)=a$, for all $a\in A$,
  \item $\cE_x(\phi\circ f)=\phi(\cE_xf)$, for all $f\in\A$ and $\phi\in\mis(A)$,
  \item If $\A$ is admissible (with $\fA=C(X)\cap\A$) then $\cE_x|_\fA$ is a character of
  $\fA$, namely, the evaluation character $\e_x$.
\end{itemize}

We now introduce the class of all homomorphisms from $\A$ into $A$ having the
same properties as the point evaluation homomorphisms $\cE_x$ ($x\in X)$.

\begin{definition}
\label{dfn:vector-valued-character}
  Let $\A$ be an admissible $A$-valued function algebra on $X$.
  A homomorphism $\Psi:\A\to A$ is called an \emph{$A$-character} if $\Psi(\U)=\U$
  and $\phi(\Psi f)=\Psi(\phi\circ f)$, for all $f\in \A$ and $\phi\in\mis(A)$.
  The set of all $A$-characters of $\A$ is denoted by $\mis_A(\A)$.
\end{definition}

That every $A$-character $\Psi:\A\to A$ satisfies $\Psi(a)=a$, for all $a\in A$, is easy to
see. In fact, since $\phi(\Psi(a))=\Psi(\phi(a))=\phi(a)$, for all $\phi\in\mis(A)$,
and $A$ is semisimple, we get $\Psi(a)=a$.


\begin{proposition}
  Let $\Psi:\A\to A$ be a linear operator such that $\Psi(\U)=\U$ and $\phi(\Psi f)=\Psi(\phi\circ f)$,
  for all $f\in \A$ and $\phi\in\mis(A)$. Then, the following are equivalent:
  \begin{enumerate}[\upshape(i)]
    \item \label{item:Psi-is-vector-valued-character}
    $\Psi$ is an $A$-character,

    \item \label{item:Psi(Inv(cA))neq0}
    $\Psi(f)\neq\z$, for every $f\in\Inv(\A)$,

    \item \label{item:Psi(Inv(fA))neq0}
    $\Psi(f)\neq\z$, for every $f\in\Inv(\fA)$,

    \item \label{item:psi-is-in-mis(fA)}
    if $\psi=\Psi|_\fA$, then $\psi\in\mis(\fA)$.
  \end{enumerate}
\end{proposition}

\begin{proof}
  $\eqref{item:Psi-is-vector-valued-character} \Rightarrow
  \eqref{item:Psi(Inv(cA))neq0} \Rightarrow \eqref{item:Psi(Inv(fA))neq0}$ is clear.
  The implication $\eqref{item:Psi(Inv(fA))neq0} \Rightarrow \eqref{item:psi-is-in-mis(fA)}$
  follows from \cite[Theorem 10.9]{Rudin-FA}.
  To prove $\eqref{item:psi-is-in-mis(fA)}\Rightarrow\eqref{item:Psi-is-vector-valued-character}$,
  take $f,g\in \A$. For every $\phi\in\mis(A)$, we have
  \[
    \phi(\Psi(fg))=\Psi(\phi\circ fg)=\psi((\phi\circ f)(\phi\circ g))
    =\psi(\phi\circ f)\psi(\phi\circ g)=\phi(\Psi(f)\Psi(g)).
  \]
  Since $A$ is semisimple, we get $\Psi(fg)=\Psi(f)\Psi(g)$.
\end{proof}

Every $A$-character is automatically continuous by Johnson's theorem \cite{Johnson-U-of-N}.
If $A$ is a uniform algebra and $\A$ is an $A$-valued uniform algebra,
we even have $\|\Psi\|=1$, for all $A$-character $\Psi$.

\begin{proposition}
\label{prop:Psi1=Psi2-iff-...}
  Let $\Psi_1$ and $\Psi_2$ be $A$-characters on $\A$, and set $\psi_1=\Psi_1|_\fA$ and
  $\psi_2=\Psi_2|_\fA$. The following are equivalent:
  \begin{enumerate}[\upshape(i)]
    \item \label{item:Psi1=Psi2} $\Psi_1=\Psi_2$,

    \item \label{item:ker Psi1=ker Psi2} $\ker\Psi_1=\ker\Psi_2$,

    \item \label{item:ker psi1=ker psi2} $\ker\psi_1=\ker\psi_2$,

    \item \label{item:psi=psi'} $\psi_1=\psi_2$.
  \end{enumerate}
\end{proposition}

\begin{proof}
  The implication $\eqref{item:Psi1=Psi2} \Rightarrow \eqref{item:ker Psi1=ker Psi2}$ is
  obvious. The implication $\eqref{item:ker Psi1=ker Psi2} \Rightarrow \eqref{item:ker psi1=ker psi2}$,
  follows from the fact that $\ker\psi_i=\ker\Psi_i \cap \fA$, for $i=1,2$.
  The implication $\eqref{item:ker psi1=ker psi2} \Rightarrow \eqref{item:psi=psi'}$
  follows form \cite[Theorem 11.5]{Rudin-FA}. Finally, if we have \eqref{item:psi=psi'},
  then, for every $f\in \A$,
  \[
    \phi(\Psi_1(f))=\psi_1(\phi(f))=\psi_2(\phi(f))=\phi(\Psi_2(f)) \quad (\phi\in\mis(A)).
  \]
  Since $A$ is semisimple, we get $\Psi_1(f)=\Psi_2(f)$, for all $f\in\A$.
\end{proof}

\begin{definition}
  Let $\A$ be an admissible $A$-valued function algebra on $X$.
  Given a character $\psi\in\mis(\fA)$, if there exists an $A$-character $\Psi$ on $\A$
  such that $\Psi|_\fA=\psi$, then we say that $\psi$ \emph{lifts} to the $A$-character
  $\Psi$.
\end{definition}

Proposition \ref{prop:Psi1=Psi2-iff-...} shows that,
if $\psi\in\mis(\fA)$ lifts to $\Psi_1$ and $\Psi_2$, then $\Psi_1=\Psi_2$.
For every $x\in X$, the unique $A$-character to which the evaluation character $\e_x$
lifts is the evaluation homomorphism $\cE_x$. In the following, we
investigate conditions under which every character $\psi\in\mis(\fA)$ lifts to
some $A$-character $\Psi$. To proceed, we need some
definitions, notations and auxiliary results.

\bigskip
Suppose $\A$ is an admissible Banach $A$-valued function algebra on $X$ with
$\fA=C(X)\cap \A$. Then $\fA$ is a Banach function algebra. Since Banach function algebras
are semisimple, by Johnson's theorem \cite{Johnson-U-of-N},
for every $\phi\in\mis(A)$, the composition operator
$C_\phi:\A\to\fA$, defined by $C_\phi(f)=\phi\circ f$, is a continuous homomorphism.
Hence, for some constant $M_\phi$, we have
\begin{equation}\label{eqn:norm(phi o f)<=Mphi norm(f)}
  \norm{\phi\circ f} \leq M_\phi \norm{f} \quad (f\in \A).
\end{equation}

Now, for every $f\in\A$ consider the function $\tilde f:\mis(A)\to\fA$,
$\phi\mapsto \phi\circ f$. Set $\cX=\mis(A)$ and
$\tilde\A=\set{\tilde f:f\in\A}$. Suppose every $\tilde f$ is continuous,
that is, if $\phi_\alpha\to\phi$ in the Gelfand topology of $\mis(A)$,
then $\phi_\alpha\circ f\to\phi\circ f$ in $\fA$ (this is the
case for uniform algebras; see Corollary \ref{cor:tilde-cA-is-fA-valued-uniform-algebra}).
Then $\tilde\A$ is an $\fA$-valued function algebra on $\cX$.
In Theorem \ref{thm:every-psi-lifts-to-Psi-iff-every-f-extends-to-F},
we will discuss conditions under which $\tilde\A$ is admissible. Since
every composite operator $C_\phi:\A \to \fA$, $f\mapsto \phi\circ f$,
with $\phi\in\mis(A)$, is continuous,
we see that the inclusion map  $\tilde\A \hookrightarrow\fA^\cX$ is continuous,
where $\fA^\cX$ is given the cartesian product topology.
The Closed Graph Theorem then implies that
the inclusion map $\tilde\A \hookrightarrow C(\cX,\fA)$
is continuous so that, for some constant $M$, we have
\begin{equation}\label{eqn:norm(phi o f)<=M norm(f)}
  \norm{\phi\circ f} \leq M \norm{f} \quad (\phi\in\mis(A),\, f\in\A).
\end{equation}

Note that $\phi\circ f\in C(X)$ and $\norm{\phi\circ f}_X \leq \|\phi\| \norm{f}_X$,
for every $\phi\in A^*$, $f\in\A$. We extend $\tilde f$ to a mapping from $A^*$
to $C(X)$, and we still denoted this extension is by $\tilde f$.

\begin{proposition}
\label{prop:continuity-of-tilde f}
  With respect to the w*-topology of $A^*$ and the uniform topology of $C(X)$,
  every mapping $\tilde f:A^*\to C(X)$ is continuous on bounded subsets of $A^*$.
\end{proposition}

\begin{proof}
  Let $\set{\phi_\alpha}$ be a net in $A^*$ that converges, in the w*-topology,
  to some $\phi_0\in A^*$ and suppose $\|\phi_\alpha\|\leq M$, for all $\alpha$.
  Take $\e>0$ and set, for every $x\in X$,
  \[
    V_x=\set{s\in X: \norm{f(s)-f(x)}<\e}.
  \]

  Then $\set{V_x:x\in X}$ is an open covering of the compact space $X$. Hence, there
  exist finitely many points $x_1,\dotsc,x_n$ in $X$ such that
  $X\subset V_{x_1}\cup \dotsb \cup V_{x_n}$. Set
  \[
    U_0 =
     \bigset{\phi\in A^*: \bigl|\phi\bigprn{f(x_i)}-\phi_0\bigprn{f(x_i)}\bigr|<\e,\,
      1\leq i \leq n}.
  \]

  The set $U_0$ is an open neighbourhood of $\phi_0$ in the w*-topology.
  Since $\phi_\alpha\to \phi_0$, there exists $\alpha_0$ such that
  $\phi_\alpha\in U_0$ for $\alpha\geq \alpha_0$.
  If $x\in X$, then $\norm{f(x)-f(x_i)}<\e$, for some $i\in\set{1,\dotsc,n}$,
  and thus, for $\alpha\geq \alpha_0$,
  \begin{align*}
    |\phi_\alpha\circ f(x)-\phi_0\circ f(x)|
     & \leq |\phi_\alpha(f(x))-\phi_\alpha(f(x_i))| + |\phi_\alpha(f(x_i))-\phi_0(f(x_i))| \\
     &  \quad + |\phi_0(f(x_i))-\phi_0(f(x))|\\
     & <M\e+\e+\|\phi_0\|\e.
  \end{align*}
  Since $x\in X$ is arbitrary, we get
  $\norm{\phi_\alpha\circ f-\phi_0\circ f}_X\leq \e(M+\|\phi_0\|+1)$, for $\alpha\geq \alpha_0$.
\end{proof}

Since $\mis(A)$ is a bounded subset of $A^*$, we get the following result for
uniform algebras.

\begin{corollary}
\label{cor:tilde-cA-is-fA-valued-uniform-algebra}
  Let $\A$ be an admissible $A$-valued uniform algebra on $A$.
  Then $\tilde f\in C(\mis(A),\fA)$, for every $f\in \A$, and
  $\tilde\A=\set{\tilde f:f\in\A}$ is an $\fA$-valued uniform algebra on $\mis(A)$.
\end{corollary}

\begin{example}
  Let $K_1$ and $K_2$ be compact subsets of the complex plane $\C$. Let $\fA=P(K_1)$ and
  $A=R(K_2)$. It is well-known that $\mis(R(K_2))=K_2$. Set
  \[
    \A=P(K_1,R(K_2))=P(K_1) \ptp_\e R(K_2).
  \]
   Every $f\in P(K_1,R(K_2))$ induces a map $\tilde f:K_2\to P(K_1)$ given by
   \[
     \tilde f(z_2)(z_1)=f(z_1)(z_2).
   \]
   We see that $\tilde \A=R(K_2,P(K_1))=R(K_2)\ptp_\e P(K_1)$.
\end{example}

\begin{lemma}
\label{lem:C(mis(A))-cap-tilde-cA=hat-A}
  For every $f\in\A$, if $\tilde f:\mis(A)\to\fA$
  is scalar-valued, then $f$ is a constant function and, therefore,
  $\tilde f = \hat a$, for some $a\in A$.
\end{lemma}

\begin{proof}
  Fix a point $x_0\in A$ and let $a=f(x_0)$. The function $\tilde f$ being scalar-valued means
  that, for every $\phi\in\mis(A)$, there is a complex number $\lambda$ such that
  $\tilde f(\phi)=\phi\circ f=\lambda$. This means that $\phi\circ f$ is a constant function
  on $X$ so that
  \begin{equation}\label{eqn:phi o f is constant}
    \phi(f(x))=\phi(f(x_0))=\phi(a) \quad (x\in X).
  \end{equation}
  Since $A$ is semisimple and \eqref{eqn:phi o f is constant} holds for every $\phi\in \mis(A)$,
  we must have $f(x)=a$, for all $x\in X$. Thus, $\tilde f =\hat a$.
\end{proof}

We are now ready to state and prove the main result of the section.

\begin{theorem}
\label{thm:every-psi-lifts-to-Psi-iff-every-f-extends-to-F}
  Let $\A$ be an admissible Banach $A$-valued function algebra on $X$,
  and let $E$ be the linear span of $\mis(A)$ in $A^*$.
  The following statements are equivalent.
  \begin{enumerate}[\upshape(i)]
    \item \label{item:g-is-continuous}
    for every $\psi\in\mis(\fA)$ and $f\in\A$, the mapping $g:E\to\C$, defined
    by $g(\phi)=\psi(\phi\circ f)$, is continuous with respect to the w*-topology of $E$;

    \item \label{item:every-psi-lifts-to-Psi}
    every $\psi\in\mis(\fA)$ lifts to an $A$-character $\Psi:\A\to A$;

    \item \label{item:every-f-extends-to-F}
    every $f\in \A$ has a unique extension $F:\mis(\fA)\to A$ such that
    \[
      \phi(F(\psi)) = \psi(\phi\circ f) \quad (\psi\in \mis(\fA),\, \phi\in\mis(A));
    \]
  \end{enumerate}
  Moreover, if the functions $\tilde f:\mis(A)\to\fA$, where $f\in\A$, are all continuous
  so that $\tilde\A=\set{\tilde f:f\in\A}$ is
  an $\fA$-valued function algebra on $\mis(A)$, then the above statements are
  equivalent to
  \begin{enumerate}[\upshape(i)]\setcounter{enumi}{3}
    \item \label{item:tilde cA is admissible}
    $\tilde \A$ is admissible.
  \end{enumerate}
\end{theorem}

\begin{proof}
  $\eqref{item:g-is-continuous}\Rightarrow\eqref{item:every-psi-lifts-to-Psi}$:
  Fix $\psi\in \mis(\fA)$ and $f\in\A$, and consider the function $g:E\to\C$ defined by
  $g(\phi)=\psi(\phi\circ f)$. Since $\phi\circ f\in\fA$, for every $\phi\in E$, we see that
  $g$ is a well-defined linear functional on $E$. Endowed with the w*-topology, $A^*$
  is a locally convex space with $A$ as its dual. Since $g$ is w*-continuous,
  by the Hahn-Banach extension theorem for locally convex spaces,
  \cite[Theorem 3.6]{Rudin-FA}, there is a w*-continuous linear functional $G$ on $A^*$
  that extends $g$. Hence $G=\hat a$, for some $a\in A$, and
  since $A$ is semisimple, $a$ is unique. Now, define $\Psi(f)=a$. Then
  \[
   \phi(\Psi(f))=\hat a(\phi) = g(\phi) = \psi(\phi\circ f)
   \quad(\phi\in\mis(A)).
  \]

  It is easily seen that $\Psi:\A\to A$ is an $A$-character and $\Psi|_\fA=\psi$.

  $\eqref{item:every-psi-lifts-to-Psi}\Rightarrow\eqref{item:every-f-extends-to-F}$:
  Fix $f\in \A$ and define $F:\mis(\fA)\to A$ by $F(\psi)=\Psi(f)$, where $\Psi$ is
  the unique $A$-character of $\A$ to which $\psi$ lifts. Considering the identification
  $x\mapsto \e_x$ and the fact that each $\e_x$ lifts to $\cE_x:f\mapsto f(x)$, we get
  \[
    F(x)=F(\e_x)=\cE_x(f)=f(x) \quad (x\in X).
  \]
  So, $F|_X=f$. Also, $\phi(F(\psi))=\phi(\Psi(f))=\psi(\phi\circ f)$, for all $\phi\in\mis(A)$
  and $\psi\in\mis(\fA)$.

  $\eqref{item:every-f-extends-to-F}\Rightarrow\eqref{item:g-is-continuous}$:
  Fix $\psi\in\mis(\fA)$ and $f\in\A$, and put $a=F(\psi)$, where $F$ is the unique extension
  of $f$ to $\mis(\fA)$ given by \eqref{item:every-f-extends-to-F}.
  Then $g(\phi)=\hat a(\phi)$, for every $\phi\in E$, which is, obviously,
  a continuous function with respect to the w*-topology of $E$.

  Finally, suppose $\tilde\A=\set{\tilde f:f\in\A}$ is an $\fA$-valued function algebra on $\mis(A)$.
  We prove $\eqref{item:every-psi-lifts-to-Psi} \Rightarrow \eqref{item:tilde cA is admissible}
  \Rightarrow \eqref{item:g-is-continuous}$. If $\A$ satisfies \eqref{item:every-psi-lifts-to-Psi},
  then
  \[
    \psi\circ \tilde f=\wh{\Psi(f)}\in\hat A \subset \tilde \A,
    \quad (f\in \A,\,\psi\in\mis(\fA)).
  \]
  This means that $\set{\psi\circ \tilde f: \psi\in\mis(\fA),\, f\in\A}\subset \tilde \A$
  and thus $\tilde\A$ is admissible. Conversely, suppose $\tilde\A$ is admissible.
  Then, for every $\psi\in\mis(\fA)$ and $f\in \A$, the function $\psi\circ \tilde f$ is
  a complex-valued function in $\tilde \A$. By Lemma \ref{lem:C(mis(A))-cap-tilde-cA=hat-A},
  $\psi\circ \tilde f$ belongs to $\hat A$ so that $\psi\circ \tilde f=\hat a$, for some $a\in A$.
  Hence, $g(\phi)=\hat a(\phi)$, for every $\phi\in E$, and $g$ is w*-continuous.
\end{proof}

In the following, we regard $A$ and $\fA$ as subalgebras of $C(\mis(A))$ and
$C(\mis(\fA))$, respectively, through the Gelfand transform.

\begin{proposition}
\label{prop:F-is-cnts-iff-tilde f-is-cnts}
  Suppose some $f\in\A$ extends to a function $F$ given by
  Theorem $\ref{thm:every-psi-lifts-to-Psi-iff-every-f-extends-to-F}$. Then
  $F:\mis(\fA)\to C(\mis(A))$ is continuous if and only if
  $\tilde f:\mis(A)\to C(\mis(\fA))$ is continuous.
\end{proposition}

\begin{proof}
  First, assume that $\tilde f:\mis(A)\to C(\mis(\fA))$.
  Given $\e>0$, since $\mis(A)$ is compact, there exist $\phi_1,\dotsc,\phi_n$ in $\mis(A)$
  such that if $\phi\in\mis(A)$ then $\|\wh{\phi\circ f}-\wh{\phi_i \circ f}\|\leq \e/3$,
  for some $\phi_i$. Given $\psi_0\in \mis(\fA)$, define a neighbourhood $V_0$
  of $\psi_0$ as follows:
  \[
    V_0 = \set{\psi\in \mis(\fA): |\psi(\phi_i\circ f)-\psi_0(\phi_i\circ f)|<\e/3,\, 1\leq i \leq n}.
  \]

  \noindent
  Now, let $\psi\in V_0$. For every $\phi\in \mis(A)$, there is $\phi_i$ such that
  \begin{align*}
    |\phi(F(\psi))-\phi(F(\psi_0))|
     & = |\psi(\phi\circ f)-\psi_0(\phi\circ f)| \\
     & \leq |\psi(\phi\circ f)-\psi(\phi_i\circ f)|
             + \e/3 + |\psi_0(\phi_i\circ f)-\psi_0(\phi\circ f)|\\
     & \leq 2\|\wh{\phi\circ f}-\wh{\phi_i \circ f}\|+\e/3 \leq \e.
  \end{align*}

  Since $\phi\in\mis(A)$ is arbitrary, we get $\|\wh{F(\psi)}-\wh{F(\psi_0)}\|\leq \e$, for $\psi\in V_0$.

  Conversely, suppose $F:\mis(\fA)\to C(\mis(A))$.
  Given $\e>0$, there exist $\psi_1,\dotsc,\psi_m$ in $\mis(\fA)$ such that if $\psi\in\mis(\fA)$
  then $\|\wh{F(\psi)}-\wh{F(\psi_j)}\|\leq \e/3$, for some $\psi_j$.
  Take $a_j=F(\psi_j)$, for $j=1,\dotsc,m$. Given $\phi_0\in \mis(A)$, take a neighbourhood $U_0$
  of $\phi_0$ as follows:
  \[
    U_0 = \set{\phi\in \mis(A): |\phi(a_j)-\phi_0(a_j)|<\e/3,\, 1\leq j \leq m}.
  \]
  The rest of the proof is similar to the previous part.
\end{proof}

\begin{theorem}
\label{thm:main-thm-holds-for-UA}
  Let $\A$ be an admissible Banach $A$-valued function algebra on $X$ and
  $\fA=C(X)\cap \A$. Suppose $\|\hat f\|=\|f\|_X$, for all $f\in \fA$.\
  Then the mapping $g$ in Theorem $\ref{thm:every-psi-lifts-to-Psi-iff-every-f-extends-to-F}$
  is continuous and thus every character $\psi:\fA\to \C$ lifts to some
  $A$-character $\Psi:\A\to A$. In particular, if $\fA$ is a uniform algebra,
  then $\A$ satisfies all conditions
  in Theorem $\ref{thm:every-psi-lifts-to-Psi-iff-every-f-extends-to-F}$.
\end{theorem}

\begin{proof}
  Take $\psi\in\mis(\fA)$ and define $g:\mis(A)\to \C$ as
  in Theorem $\ref{thm:every-psi-lifts-to-Psi-iff-every-f-extends-to-F}$.
  Since $\|\hat f\|=\|f\|_u$, for every $f\in \fA$,
  $\psi$ is a continuous functional on $(\fA,\enorm_u)$; see also \cite{Honary}.
  By the Hahn-Banach theorem, $\psi$ extends to a continuous linear functional
  $\bar\psi$ on $C(X)$. This, in turn, implies that $g$ extends to a linear functional
  $\bar g:A^*\to \C$ defined by $\bar g(\phi)=\bar\psi(\phi\circ f)$.
  By Proposition \ref{prop:continuity-of-tilde f}, the mapping $\tilde f:A^*\to C(X)$,
  $\phi\mapsto \phi\circ f$, is w*-continuous on bounded subsets of $A^*$. Hence, the linear
  functional $\bar g$ is w*-continuous on bounded subsets of $A^*$. Since $A$ is a Banach
  space, Corollary 3.11.4 in \cite{Horvath} shows that $\bar g$ is w*-continuous on $A^*$.
\end{proof}

\begin{remark}
  Let $\A$ be an admissible $A$-valued uniform algebra with $\fA=\A\cap C(X)$.
  By Theorem \ref{thm:main-thm-holds-for-UA}, every $f\in \A$ extends to a
  function $F:\mis(\fA)\to A$. If, in addition, $A$ is a uniform algebra so that
  $\|a\|=\|\hat a\|_{\mis(A)}$, by Proposition \ref{prop:F-is-cnts-iff-tilde f-is-cnts},
  this extension $F$ is continuous and one can see the following maximum principle
  \[
     \|F\|_{\mis(\fA)}=\|F\|_X = \|f\|_X.
  \]
\end{remark}

\begin{example}
  Let $K$ be a compact plane set. Then every $f\in P(K,A)$ has an extension
  $F:\hat K\to A$ which belongs to $P(\hat K,A)$, and very $\lambda \in \hat K$
  induces an $A$-character $\cE_\lambda:P(K,A)\to A$ given by
  $\cE_\lambda(f)=F(\lambda)$.
\end{example}

%
\begin{corollary}
  If $\A$ is an admissible $A$-valued uniform algebra on $X$, then
  $\tilde \A$ is an admissible $\fA$-valued uniform algebra
  on $\mis(A)$.
\end{corollary}

\begin{remark}
  When $\fA$ is uniformly closed in $C(X)$, every linear functional $\psi\in \fA^*$
  lifts to some bounded linear operator $\Psi:\A\to A$ with the property that
  $\phi(\Psi f)=\psi(\phi\circ f)$, for all $\phi\in A^*$. In fact, by the Hahn-Banach
  theorem, $\psi$ extends to a linear functional $\bar \psi\in C(X)^*$.
  Then, by the Riesz representation theorem, there is a complex Radon measure $\mu$
  on $X$ such that
  \[
   \psi(f) = \int_X f d\mu \quad (f\in \fA).
  \]
  Using \cite[Theorem 3.7]{Rudin-FA}, one can define $\Psi(f)=\int_X f d\mu$, for every
  $f\in \A$, such that
  \[
    \phi(\Psi(f))=\int_X (\phi \circ f) d\mu = \psi(\phi\circ f)
    \quad (f\in \A,\, \phi\in A^*).
  \]
\end{remark}



\bibliographystyle{amsplain}

\end{document}